\documentclass[12pt,letterpaper,titlepage]{amsart}
\usepackage{amsmath, amssymb, amsthm, amsfonts,amscd,amsaddr,xcolor}
\usepackage[a4paper, portrait, textwidth=400pt, textheight=650pt, tmargin=100pt, lmargin=100pt]{geometry}
\usepackage[english]{babel}
\usepackage[utf8]{inputenc}
\usepackage[T1]{fontenc}

\theoremstyle{plain}
\newtheorem{theorem}{Theorem}[section]

\newtheorem{lemma}[theorem]{Lemma}
\theoremstyle{definition}

\numberwithin{equation}{section}

\newcommand{\C}{\mathbb{C}}

\newcommand{\N}{\mathbb{N}}

\newcommand{\R}{\mathbb{R}}
\newcommand{\Z}{\mathbb{Z}}

\newcommand{\Cc}{\mathcal{C}}

\newcommand{\Ec}{\mathcal{E}}

\newcommand{\Hc}{\mathcal{H}}

\newcommand{\Vc}{\mathcal{V}}

\def\gf{\mathfrak{g}}

\newcommand{\Hom}{\operatorname{Hom}}

\newcommand{\SO}{\operatorname{SO}}

\newcommand{\so}{\mathfrak{so}}

\newcommand{\diag}{\operatorname{diag}}

\newcommand{\re}{\operatorname{Re}}

\newcommand{\even}{\operatorname{even}}
\newcommand{\odd}{\operatorname{odd}}
\newcommand{\parity}{\operatorname{parity}}
\newcommand{\temp}{\operatorname{temp}}

\def\1{{\bf1}}

\title[Discrete series embeddings]
{Discrete series representations with non-tempered embedding}

\subjclass[2000]{}
\begin{document}
\date{March 2021}

\begin{abstract}
We give an example of a semisimple symmetric space $G/H$ and an irreducible
representation of $G$ which has multiplicity $1$ in $L^2(G/H)$
and multiplicity $2$ in $C^\infty(G/H)$.
\end{abstract}


\author[Kr\"otz]{Bernhard Kr\"{o}tz}
\email{bkroetz@gmx.de}
\address{Institut f\"ur Mathematik, Universit\"at Paderborn,\\ Warburger Stra\ss e 100,
33098 Paderborn}

\author[Kuit]{Job J.~Kuit}
 \email{jobkuit@math.upb.de}
\address{Institut f\"ur Mathematik, Universit\"at Paderborn, \\Warburger Stra\ss e 100,
33098 Paderborn}

\author[Schlichtkrull]{Henrik Schlichtkrull}
\email{schlicht@math.ku.dk}
\address{University of Copenhagen, Department of Mathematics\\Universitetsparken 5,
DK-2100 Copenhagen \O}

\maketitle
\section{Introduction}
Let $G$ be a real reductive group and $X=G/H$ an attached symmetric space. Let further $V$
be a Harish-Chandra module and $V^\infty$ its smooth Fr\'echet completion of moderate growth. We say that
$V$ is $H$-spherical provided
$$\Hom_G (V^\infty, C^\infty(X))\neq 0\, .$$
By Frobenius reciprocity we have $ \Hom_G (V^\infty, C^\infty(X))\simeq (V^{-\infty})^H$ with $V^{-\infty}= (V^\infty)'$ the strong dual
of $V^\infty$. We recall that the full multiplicity space $(V^{-\infty})^H$ is finite dimensional.

\par Inside of $C^\infty(X)$ we find natural invariant subspaces: the Harish-Chandra Schwartz space $\Cc(X)\subset L^2(X)\cap C^\infty(X)$ and the
space $C_{\temp}^\infty(X)$ of tempered functions which lie in $L^{2+\epsilon}(X)$ for all $\epsilon >0$. Accordingly
we obtain subspaces
$$(V^{-\infty})_{\rm temp}^H, (V^{-\infty})_{\rm disc}^H\subset  (V^{-\infty})^H$$
 corresponding to
 $ \Hom_G (V^\infty, C^\infty_{\rm temp} (X))$ and  $ \Hom_G (V^\infty, \Cc(X))$ respectively.
 The objective of this paper is to provide an example where
 $$0\neq (V^{-\infty})_{\rm disc}^H=(V^{-\infty})_{\rm temp}^H \subsetneq (V^{-\infty})^H\, .$$
To be more specific this happens for $X$ the $n$-dimensional one-sheeted hyperboloid which is homogeneous for the connected
 Lorentzian group $G=\SO_0(n,1)$. Let us briefly introduce the standard notions.
 \subsection{Notation} Let $n\ge 3$ and let $G=\SO_0(n,1)$ be the identity component of the special Lorentz group $\SO(n,1)$.
We denote by $H$ the stabilizer of
$$x_0:=(1,0,\dots,0)\in \R^{n+1}$$ in $G$.
The entry in the lower right corner of any matrix in $\SO(n,1)$ is non-zero, and $\SO_0(n,1)$
consists of those matrices for which this entry is positive. From this fact we see that
$H$ is the connected subgroup
$$
H=\begin{pmatrix} 1&0\\0&\SO_0(n\!-\!1,1)\end{pmatrix}\subset G.
$$

The group $G$ acts
transitively on the hyperboloid
$$ X:=\{x\in\R^{n+1}\mid x_1^2+\dots +x_n^2-x_{n+1}^2=1\},$$
and
the homogeneous space
$$X=G/H=\SO_0(n,1)/\SO_0(n-1,1)$$
is a symmetric space. The corresponding involution $\sigma$ of $G$ is
given by conjugation with  the diagonal matrix $\diag(-1,1,\dots,1)$. The subgroup $G^\sigma$ of
$G$ of $\sigma$-fixed elements is the stabilizer of $\R x_0$.
This subgroup has two components, one of which is $H$.
For our purpose it is important to use $H$ rather than $G^\sigma$. The pairs
$(G,H)$ and $(G,G^\sigma)$ differ by the fact that $(G,G^\sigma)$ is a Gelfand pair,
whereas $(G,H)$ is not. In fact it has been shown by
van Dijk \cite{vanDijk} that $X$ is the only symmetric space
of rank one, which is obtained as the homogeneous space of a Gelfand pair.

The regular representation of $G$ on $C^\infty(X)$ decomposes as the direct sum
$$C^\infty(X)=C_{\even}^\infty(X)\oplus C_{\odd}^\infty(X)$$
of the invariant subspaces
of functions that are even or odd with respect to
the $G$-equivariant symmetry $x\mapsto -x$.
The restriction of the regular representation  to $C_{\even}^\infty(X)$
is isomorphic to the regular representation on $C^\infty(G/G^\sigma)$,
and the non-Gelfandness of $(G,H)$ is therefore caused by the presence of the
odd functions.

\subsection{Main results}

The Lorentzian manifold $X$ carries the $G$-invariant Laplace-Beltrami operator $\Delta$, which is
obtained as the radial part of
$$\Box:=-\frac{\partial^2}{\partial x_1^2}-\dots-\frac{\partial^2}{\partial x_n^2}+\frac{\partial^2}{\partial x_{n+1}^2}$$
on $\R^{n+1}$. That is, for $f\in C^\infty(X)$ we define $\Delta f\in C^\infty(X)$
as $(\Box \tilde f)|_ X$ where $\tilde f$ is any extension of $f$
to a function, homogeneous of degree $0$, on some open neighborhood of $X$ in $\R^{n+1}$.

Let
$$\rho:=\tfrac12(n-1)$$
and for each $\lambda\in\C$ let $\Ec_\lambda(X)$ be the eigenspace
$$\Ec_\lambda(X)=\{f\in C^\infty(X)\mid \Delta f=(\lambda^2-\rho^2)f\}.$$
The Laplace-Beltrami operator is a scalar multiple of the Casimir element
associated to the Lie group $G$, and hence every irreducible subspace $\Vc$ of $C^\infty(X)$ is contained
in $\Ec_\lambda(X)$ for some $\lambda\in\C$. Apart from the sign, the scalar $\lambda$ is uniquely determined
by the infinitesimal character of $\Vc$. Conversely, since $X=G/H$ is rank one, $\pm\lambda$ determines
the infinitesimal character.

Let
$$\Ec^{\even}_\lambda(X)=\Ec_\lambda(X)\cap C_{\even}^\infty(X),\quad \Ec^{\odd}_\lambda(X)=\Ec_\lambda(X)\cap C_{\odd}^\infty(X).$$
We can now state our main results.

The manifold $X$ carries a $G$-invariant measure, which is unique up to scalar multiplication.
We denote by $L^2(X)$ the associated $G$-invariant space of square integrable functions.

\begin{theorem} \label{thm discrete series}
Let $\lambda\in\C$ with $\re\lambda> 0$. The intersections $\Ec^{\even}_\lambda(X) \cap L^2(X)$ and
$\Ec^{\odd}_\lambda(X) \cap L^2(X)$ are either zero or irreducible. Moreover,
\begin{align*}
\Ec^{\even}_\lambda(X) \cap L^2(X)\neq 0\, &\Leftrightarrow\,  \lambda\in \rho+1+2\Z\\
\Ec^{\odd}_\lambda(X) \cap L^2(X)\neq 0\, &\Leftrightarrow \,\lambda\in \rho+2\Z .
\end{align*}
\end{theorem}

\begin{theorem} \label{thm small lambda}
For every $0<\lambda<\rho$ with $\lambda\in \rho-\N$ the $G$-representations
$\Ec^{\even}_\lambda(X)$
and $\Ec^{\odd}_\lambda(X)$ are irreducible and infinitesimally equivalent. Moreover in this case,
\begin{enumerate}
\item
if $\lambda-\rho$ is even then $\Ec^{\odd}_\lambda(X)$ is contained in $L^2(X)$
and $\Ec^{\even}_\lambda(X)$ is non-tempered, and
\item
if $\lambda-\rho$ is odd then $\Ec^{\even}_\lambda(X)$ is contained in $L^2(X)$
and $\Ec^{\odd}_\lambda(X)$ is non-tempered.
\end{enumerate}
\end{theorem}

For $n\ge 4$ we have $\rho>1$ and it follows that there exists at least one discrete series representation
for $X=G/H$ which has multiplicity
$1$ in $C^\infty_{\temp}(X)$, but for which the underlying Harish-Chandra module has multiplicity $2$ in $C^\infty(X)$.

The complete Plancherel decomposition for $\SO_0(n,1)/\SO_0(n-1,1)$ is
given in \cite{LNR}. However, this is not needed for the proof of our theorems.

\section{Proof of the main results}
The proof of the two main theorems is divided into several parts. We begin with the analysis on  $K$-types.
\subsection{$K$-types}\label{subsection K-types}

Let $K\subset G$ be the stabilizer of $e_{n+1}=(0,\dots,0,1)\in\R^{n+1}$,
then $K\simeq \SO(n)$ is a maximal compact subgroup of $G$.

We are going to use the diffeomorphism $S^{n-1}\times\R\overset{\sim}{\to} X$  given by
\begin{equation}\label{eq polar coords}
(y,t)\mapsto (y_1\cosh t,\dots,y_n\cosh t,\sinh t)\in X
\end{equation}
where $y=(y_1,\dots,y_n)\in S^{n-1}$ and $t\in\R$. With the natural action of
$\SO(n)$ on $S^{n-1}$ the parameter dependence on $y$ is $K$-equivariant.

For each $j\in\N_0$ we denote by $\Hc_j\subset C^\infty(S^{n-1})$ the space of spherical harmonics of degree $j$.
We recall that by definition $\Hc_j$ consists of the restrictions to $S^{n-1}$ of all harmonic polynomials on $\R^n$,
homogeneous of degree $j$.
Equivalently, $\Hc_j$ can be defined as the eigenspace
$$\Hc_j:=\{h\in C^\infty(S^{n-1})\mid \Delta_K h= -j(j+n-2) h\},$$
where $\Delta_K$ is the angular part of the $n$-dimensional Laplacian
$$\frac{\partial^2}{\partial y_1^2}+\dots+\frac{\partial^2}{\partial y_n^2}.$$
Each $\Hc_j$
is an irreducible $\SO(n)$-invariant finite dimensional subspace of $C^\infty(S^{n-1})$,
and the sum $\oplus_{j=0}^\infty \Hc_j$ of these subspaces is dense in $C^\infty(S^{n-1})$.

It follows that the space $C_K^\infty(X)$ of $K$-finite functions $f\in C^\infty(X)$ is spanned by all
functions given in the coordinates $(y,t)$ by
$$f(y,t)=h(y)\varphi(t),$$
where $h\in C^\infty(S^{n-1})$ is a spherical harmonic, and $\varphi\in\C^\infty(\R)$. Thus
$$C_K^\infty(X)\underset{K}{\simeq} \oplus_{j=0}^\infty \big(\Hc_j\otimes C^\infty(\R)\big).$$

Being homogeneous of degree $j$, the spherical harmonics $h\in\Hc_j$ satisfy $h(-y)=(-1)^jh(y)$ for $y\in S^{n-1}$. Therefore
$$C_K^\infty(X)_{\even}\underset{K}{\simeq} \oplus_{j=0}^\infty \big(\Hc_j\otimes C^\infty_{\parity(j)}(\R)\big)$$
where $\parity(j)$ denotes the parity $\even$ or $\odd$ of $j$.
Likewise
$$C_K^\infty(X)_{\odd}\underset{K}{\simeq} \oplus_{j=0}^\infty \big(\Hc_{j}\otimes C^\infty_{\parity(j+1)}(\R)\big).$$

\subsection{Eigenspaces}\label{subsection Eigenspaces}

With respect to the coordinates \eqref{eq polar coords} on $X$ we have (see \cite[p. 455]{Rossmann})
$$
\Delta=\frac{\partial^2}{\partial t^2}+2\rho\tanh t\, \frac{\partial}{\partial t}-\frac1{\cosh^2t}\,\Delta_K.
$$
The $K$-finite
eigenfunctions for $\Delta$ belong to $C^\infty(X)$, and they can be determined as follows.
For each  $j\in\N_0$ we let
$$\Ec_{\lambda,j}(X)=\Ec_\lambda(X)\cap \big(\Hc_j\otimes C^\infty(\R)\big).$$
By separating the variables $y$ and $t$ we see that $\Ec_{\lambda,j}$ is spanned by the
functions $f(y,t)=h(y)\varphi(t)$ for which $h\in \Hc_j$
and
\begin{equation}\label{ode for varphi}
\Big(\frac{d^2}{d t^2}+2\rho\tanh t\, \frac{d}{d t}+\frac{j(j+n-2)}{\cosh^2t}\Big)\varphi=(\lambda^2-\rho^2)\varphi.
\end{equation}
This differential equation is invariant under sign change of $t$.
The solution with $\varphi(0)=1$ and $\varphi'(0)=0$ is even, and
the solution with $\varphi(0)=0$ and $\varphi'(0)=1$ is odd.
Thus the solution space decomposes as the direct sum of the
one-dimensional subspaces of even and odd solutions, and we have
$$
\Ec_{\lambda,K}^{\even}(X)=\oplus_{j=0}^\infty \,\Ec_{\lambda,j}^{\even}(X), \qquad
\Ec_{\lambda,K}^{\odd}(X)=\oplus_{j=0}^\infty \,\Ec_{\lambda,j}^{\odd}(X)
$$
where $\Ec_{\lambda,j}^{\even}(X)\underset{K}{\simeq} \Ec_{\lambda,j}^{\odd}(X)\underset{K}{\simeq} \Hc_j$ are equivalent
irreducible $K$-types for each $j$.

\subsection{Hypergeometric functions}\label{subsection Hypergeometric functions}

In fact \eqref{ode for varphi} can be transformed into a standard equation of special function theory.
We first prepare for the anticipated asymptotic behavior af $\varphi$ by
substituting $\Phi(t)=(\cosh t)^{\lambda+\rho}\,\varphi(t)$.
This leads to the following equation for
$\Phi(t)$
\begin{equation}\label{ode for Phi}
\Phi''(t)-2\lambda\tanh t\,\Phi'(t)-ab\,(1-\tanh^2\! t)\,\Phi(t)=0,
\end{equation}
where $a=\lambda+\rho+j$ and $b=\lambda-\rho+1-j$.

Next we change variables. With
$$x=\frac12(1-\tanh t)=(1+e^{2t})^{-1}\in (0,1)$$
we replace the limits $t=\infty$ and  $t=-\infty$ by $x=0$ and $x=1$, respectively.
We write $\Phi(t)=F(x)$, so that
$$\varphi(t)=(\cosh t)^{-\lambda-\rho} F((1+e^{2t})^{-1}).$$
Since  $\tanh t=1-2x$ and $1-\tanh^2\!t=4x(1-x)$ this gives
$$(x')^2F''(x)+x''F'(x)-2\lambda(1-2x)\,F'(x)x'-ab\,4x(1-x)F(x)=0,$$
and
since
$$x'=-2x(1-x), \qquad x''=4x(1-x)(1-2x),$$
we arrive at the following equation for the function $F(x)$
\begin{equation}\label{ode for F}
x(1-x)F''(x)+(\lambda+1)(1-2x) F'(x)-abF(x)=0\,.
\end{equation}

Recall the hypergeometric function $F(a,b;c;x)={}_2F_1(a,b;c;x)$, which with the notation
$(a)_m:=\prod_{k=0}^{m-1}(a+k)$ is defined by the Gauss series
$$
F(a,b;c;x)=\sum_{m=0}^\infty
\frac{(a)_m(b)_m}{m!\,(c)_m}\,x^m
$$
for
$x\in\C$ with $|x|<1$, for all $a,b,c\in\C$ except $c\in -\N_0$.
It solves Euler's hypergeometric differential equation
\begin{equation}\label{hypergeometric equation}
x(1-x)\,w''+(c-(a+b+1)x)\,w'-ab\, w =0.
\end{equation}
The function $F(a,b;c;x)$ is analytic at $x=0$ with the value $1$, and unless $c$ is an integer
it is the unique solution with this property.

With $a=\lambda+\rho+j$ and $b=\lambda-\rho+1-j$ as above we have $a+b+1=2(\lambda+1)$.
By comparing \eqref{ode for F} and \eqref{hypergeometric equation}
we conclude that for each $\lambda\notin-\N$ the function
$$\varphi_{\lambda,j}(t):=(\cosh t)^{-\lambda-\rho} \,F(\lambda+\rho+j,\lambda-\rho+1-j;1+\lambda;(1+e^{2t})^{-1})$$
solves \eqref{ode for varphi}.

\subsection{$L^2$-behavior}\label{subsection L2-behavior}
In the coordinates $(y,t)$ an invariant measure on $X$ is given by
$$\cosh^{n-1}t\,dt\,dy$$
where $dt$ and $dy$ are invariant measures on $\R$ and $S^{n-1}$, respectively.
Hence a function $f(y,t)=h(y)\varphi(t)$ is square integrable if and only if
$$\int_\R |\varphi(t)|^2 \cosh^{n-1}\!t\,dt <\infty.$$

Let $\check\varphi_{\lambda,j}(t)=\varphi_{\lambda,j}(-t)$.
By symmetry this function also solves \eqref{ode for varphi}, and it belongs to
$L^2(\R,\cosh^{n-1}\!t\,dt)$ if and only if $\varphi_{\lambda,j}$ does.

\begin{lemma}\label{lemma sq int}
Let $\re\lambda>0$ and $j\in\N_0$.
\begin{enumerate}
\item\label{lemma sq int (1)} If $j\in\lambda-\rho+\N$ then
$\varphi_{\lambda,j}\in L^2(\R,\cosh^{n-1}\!t\,dt).$
\item\label{lemma sq int (2)}
If $j\notin\lambda-\rho+\N$ then $\varphi_{\lambda,j}$ and $\check\varphi_{\lambda,j}$
are linearly independent, and for a sufficiently small $\epsilon>0$ no non-trivial linear combination belongs to
$L^{2+\epsilon}(\R,\cosh^{n-1}\!t\,dt)$.
\end{enumerate}
\end{lemma}

\begin{proof}It follows from the definition of $\varphi_{\lambda,j}(t)$ that
$$(\cosh t)^{\lambda+\rho}\, \varphi_{\lambda,j}(t)\to 1, \qquad t\to\infty.$$
Since $2\rho=n-1$ this means $\varphi_{\lambda,j}$
has
the desired $L^2$-behavior in the positive direction
for all  $\re\lambda>0$. The only issue is with the negative direction, or equivalently,
with $\check\varphi_{\lambda,j}(t)$ for $t\to\infty$.

We first consider \eqref{lemma sq int (1)}. The assumption that $j\in\lambda-\rho+\N$
implies that $b=\lambda-\rho+1-j\in -\N_0$.
The Gauss series for $F(a,b;c;x)$ terminates and defines a polynomial
when $a$ or $b$ is a non-positive integer.
In particular $F(a,b;c;x)$ is then a bounded function on $[0,1]$.
It then follows from the definition of $\varphi_{\lambda,j}$ that
$$(\cosh t)^{\lambda+\rho}\, \varphi_{\lambda,j}(t)$$
is bounded on $\R$, and hence $|\varphi_{\lambda,j}(t)|^2( \cosh t)^{2\rho}$ is integrable.
This proves \eqref{lemma sq int (1)}.

Now consider \eqref{lemma sq int (2)}.
According to Gauss (see \cite[Thm.~2.1.3]{AAR})  we have
$$\lim_{x\to1^-} (1-x)^{a+b-c}F(a,b;c;x)=A:=\frac{\Gamma(c)\Gamma(a+b-c)}{\Gamma(a)\Gamma(b)}$$
if $\re(a+b-c)>0$.
We apply this to $t\to\infty$ in
$$(\cosh t)^{\lambda+\rho} \,\check\varphi_{\lambda,j}(t)=F(\lambda+\rho+j,\lambda-\rho+1-j;1+\lambda;(1+e^{-2t})^{-1}).$$
Here $a+b-c=\lambda$ and
$$
A=\frac{\Gamma(1+\lambda)\Gamma(\lambda)}{\Gamma(\lambda+\rho+j)\Gamma(\lambda-\rho+1-j)}.
$$
Since $x=(1+e^{-2t})^{-1}$ implies $1-x=(1+e^{2t})^{-1}$ it follows that
$$
e^{-2\lambda t} (\cosh t)^{\lambda+\rho}\,\check\varphi_{\lambda,j}(t)\to  A, \qquad t\to \infty.
$$
In particular we note that $A\neq0$ if $j\notin \lambda-\rho+\N$. Under this condition
we see that no non-trivial linear combination of
$\varphi_{\lambda,j}$ and $\check\varphi_{\lambda,j}$ exhibits $L^{2+\epsilon}$-behavior
in both directions $\pm\infty$.

This proves \eqref{lemma sq int (2)} and concludes the proof of the lemma.
\end{proof}

\subsection{Parity of $\varphi_{\lambda,j}$}\label{subsection Parity of varphi}

We have seen that $\varphi_{\lambda,j}$ and $\check\varphi_{\lambda,j}$ are independent solutions to
\eqref{ode for varphi} when $j\notin\lambda-\rho+\N$. For $j\in\lambda-\rho+\N$ they are proportional,
as the following lemma shows.
For $\alpha,\beta>-1$ and $l\in\N_0$ let
$$P^{(\alpha,\beta)}_l(z):=\frac{(\alpha+1)_l}{l!}\,F(l+\alpha+\beta+1,-l;\alpha+1;\frac12(1-z))$$
be the corresponding Jacobi polynomial (see \cite[page 115]{BW}).

\begin{lemma}\label{lemma even/odd eigen}
Assume $j=\lambda-\rho+1+l\in \N_0$ where $l\in\N_0$. Then
\begin{enumerate}\itemsep=2pt
\item\label{item one of lemma even/odd eigen}
$\frac{(\lambda+1)_l}{l!}\,\varphi_{\lambda,j}(t)=(\cosh t)^{-\lambda-\rho}P^{(\lambda,\lambda)}_l(\tanh t),$
\item\label{item two of lemma even/odd eigen}
$\varphi_{\lambda,j}(-t)=(-1)^l\,\varphi_{\lambda,j}(t)$,
\end{enumerate}
for all $t\in\R$.
\end{lemma}

Note that with the repeated indices $P^{(\lambda,\lambda)}_l$ is in fact a Gegenbauer polynomial.

\begin{proof}
By definition
$$\varphi_{\lambda,j}(t)=(\cosh t)^{-\lambda-\rho}\, F(2\lambda+1+l,-l;\lambda+1;x)$$
where $x=(1+e^{2t})^{-1}=\frac12(1-\tanh t)$. The equation (\ref{item one of lemma even/odd eigen}) follows
immediately. Then (\ref{item two of lemma even/odd eigen}) follows since
a Gegenbauer polynomial is even or odd according to the parity
of its degree.
\end{proof}

\subsection{$K$-types in $L^2$}\label{subsection K-types in L2}
For $\lambda\in\C$ with $\re\lambda\ge 0$ we define
$$
D_\lambda:=\N_0 \cap (\lambda-\rho+\N)
$$
if $\lambda-\rho\in\Z$ and $\lambda> 0$, and by $D_\lambda=\emptyset$ otherwise.
Furthermore we let
$$
U_\lambda:=\bigoplus_{j\in D_\lambda} \, \big(\Hc_j\otimes \varphi_{\lambda,j}\big).
$$
It follows  from
Lemma \ref{lemma even/odd eigen} and the fact that the parity of $\Hc_j$ is $(-1)^j$ that
$U_\lambda\subset \Ec_{\lambda,K}^{\even} (X)$
if $\lambda-\rho$ is odd, and
$U_\lambda \subset \Ec_{\lambda,K}^{\odd} (X)$
if $\lambda-\rho$ is even.

\begin{lemma}\label{lemma square integrability}
For all $\lambda\in\C$ with $\re\lambda> 0$ we have
\begin{equation}\label{eq U tempered}
U_\lambda=\Ec_{\lambda,K}(X)\cap L^2(X)  =\Ec_{\lambda,K}(X)\cap C^\infty_{\temp}(X).
\end{equation}
\end{lemma}

\begin{proof} Let $\re\lambda> 0$. When $\lambda-\rho\notin\Z$
it follows immediately from
Lemma \ref{lemma sq int}(\ref{lemma sq int (2)}) that
$$\Ec_{\lambda,j}(X)\cap  C^\infty_{\temp}(X)=\{0\}$$
for each $j\in\N_0$. Since $U_\lambda=\{0\}$ in this case \eqref{eq U tempered} follows.
Assume from now on that $\lambda-\rho\in\Z$.
It then follows from
Lemma \ref{lemma sq int}(\ref{lemma sq int (1)}) that $U_\lambda\subset L^2(X)$.

To complete the proof we will find for each $j\in\N_0$
a second solution to \eqref{ode for varphi},
which is linearly independent from $\varphi_{\lambda,j}$, and which is not
tempered. There are two cases, depending on the parity of $n$.

If $n$ is even, then $\rho$, and hence also $\lambda$, is not an integer. In that case
we already have a second solution at hand, namely $\varphi_{-\lambda,j}$.
Since
$$(\cosh t)^{-\lambda+\rho}\,\varphi_{-\lambda,j}(t)\to 1, \qquad t\to\infty,$$
this function $\varphi_{-\lambda,j}$ does not belong to any
$L^p(\R,\cosh^{n-1}\!t\,dt)$ if $\re\lambda\ge\rho$.
When $0<\re\lambda<\rho$ it belongs to
$L^{2+\epsilon}$ only for $\epsilon>\frac{2\re\lambda}{\rho-\re\lambda}$.

We now assume $n$ is odd. Then $\rho$ and $\lambda$ are positive integers.
We need to find a solution
linearly independent from $F(a,b;c;x)$
to the hypergeometric equation
\eqref{hypergeometric equation} with
$$a=\lambda+\rho+j,\quad b=\lambda-\rho+1-j,\quad c=1+\lambda.$$
By the method of Frobenius one finds (see \cite[p. 5]{Norlund}) such a solution
$G(a,b,c;x)$. It has the form
$$G(x)=x^{-\lambda}\sum_{\nu=0}^\infty a_\nu x^\nu +\log x\,\sum_{\nu=0}^\infty b_\nu x^\nu$$
for some  explicit power series with $a_0=1$ and $b_0\neq 0$.
The corresponding solution to \eqref{ode for varphi} is
$$(\cosh t)^{-\lambda-\rho}\,G(a,b,c;(1+e^{2t})^{-1}).$$
It behaves like $(\cosh t)^{-\lambda-\rho}(1+e^{2t})^{\lambda}$ as $t\to\infty$ and
as before it is not tempered with respect to the invariant measure.
\end{proof}

\subsection{Irreducibility}\label{subsection Irreducibility}
Let $\lambda\in\rho+\Z$ and assume $\lambda> 0$.
It follows from Lemma \ref{lemma square integrability} that $U_\lambda$ is $(\gf,K)$-invariant.
We will prove that it is an irreducible $(\gf,K)$-module by using the infinitesimal element
\begin{equation}\label{defi T}
T=E_{n+1,1}+E_{1,n+1}\in\gf=\so(n,1),
\end{equation}
as a raising and lowering operator between the functions $\varphi_{\lambda,j}$ which generate $U_\lambda$ together
with $K$. For this we need to find the derivative of $\varphi_{\lambda,j}$.

\begin{lemma}\label{lemma derivative varphi}
Let $j=\lambda-\rho+1+l\in\N_0$ where $l\in\N_0$.
There exist constants $A_l,B_l\in\R$ such that
$$
\varphi_{\lambda,j}'=A_l\,\varphi_{\lambda,j+1}+     B_l\,\varphi_{\lambda,j-1}.
$$
Both $A_l$ and $B_l$ are non-zero, except when $l=0$ or $j=0$, in which cases only
$A_l$ is non-zero.
\end{lemma}

\begin{proof}
Recall from Lemma \ref{lemma even/odd eigen}
$$\varphi_{\lambda,j}(t)=\tfrac{l!}{(\lambda+1)_l}\,(\cosh t)^{-\lambda-\rho}\,P^{(\lambda,\lambda)}_l(\tanh t)$$
for $j=\lambda-\rho+1+l$. It follows that
$$\varphi_{\lambda,j}'(t)=\tfrac{l!}{(\lambda+1)_l}\,(\cosh t)^{-\lambda-\rho}\,
\big(\! -\!(\lambda+\rho)x\,P^{(\lambda,\lambda)}_l(x)+ (1-x^2)\,(P_l^{(\lambda,\lambda)})'(x)\big)$$
where $x=\tanh t$.

We obtain from \cite[(4.7)]{THK} that
$$
(1-x^2)\,(P_l^{(\lambda,\lambda)})'(x)=(l+2\lambda+1)\,xP_l^{(\lambda,\lambda)}(x)
    -\frac{(l+1)(l+2\lambda+1)}{l+\lambda+1}\,P_{l+1}^{(\lambda,\lambda)}(x).
$$
By \cite[(5.5.5)]{BW} the polynomials $P^{(\lambda,\lambda)}_l$ satisfy a three term recurrence relation
\begin{align*}
&(l+\lambda+1)(2l+2\lambda+1)\,xP^{(\lambda,\lambda)}_{l}(x)\\
&\quad=(l+\lambda)(l+\lambda+1)\,P^{(\lambda,\lambda)}_{l-1}(x)
+(l+1)(l+2\lambda+1)\,P^{(\lambda,\lambda)}_{l+1}(x).
\end{align*}
With this relation we can eliminate $xP^{(\lambda,\lambda)}_{l}(x)$ and
obtain $\varphi_{\lambda,j}'$ as a linear combination of $\varphi_{\lambda,j+1}$ and
$\varphi_{\lambda,j-1}$. The coefficients turn out to be
$$A_l=-\frac{(\lambda+\rho+l)(2\lambda+l+1)}{2\lambda+2l+1},\quad
B_l=\frac{l(\lambda-\rho+l+1)}{2\lambda+2l+1}.$$
All these coefficients are non-zero, except $B_l$ when $l=0$ or $\lambda-\rho+l+1=0$.
\end{proof}

Let $M= K\cap H$ be the stabilizer in $K$ of $x_0$, that is
$$M=\begin{pmatrix}1&&0\\&\SO(n-1)&\\0&&1\end{pmatrix}\subset \SO_0(n,1)=G.$$
Then $S^{n-1}\simeq K/M$. Let $h_j\in\Hc_j$ be the zonal spherical harmonic.
This is the unique function in $\Hc_j$ which is $M$-invariant and has the value
$1$ at the origin $(1,0,\dots,0)$ of $S^{n-1}$. Furthermore, let $f_{\lambda,j}\in\Ec_{\lambda}(X)$
be defined by
$$f_{\lambda,j}(y,t)=h_j(y)\varphi_{\lambda,j}(t).$$

With the element $T$ from \eqref{defi T}
the coordinates  $(y,t)$ are determined from
$$K/M\times \R \ni (kM,t) \mapsto k\exp(tT)x_0 \in X,$$
and since $T$ is centralized by $M$ the left derivative $L_T f_{\lambda,j}$ by $T$ is again $M$-invariant.
It follows that for each $j\in\N_0$, the function $L_T f_{\lambda,j}\in\Ec_{\lambda,K}(X)$ is a linear combination of the
same family of functions $f_{\lambda,\boldsymbol{\cdot}}$ in $\Ec_{\lambda, K}(X)$.
Since $h_j(x_0)=1$ for all $j$, the coefficients can be determined from the restriction to
$$\{(\cosh t,0,\dots,0,\sinh t)\mid t\in\R\}\subset X,$$
on which $L_T$ acts just by $\frac{d}{dt}$, and hence they are given by Lemma
\ref{lemma derivative varphi}. It follows immediately that $U_\lambda$ has no non-trivial $(\gf,K)$-invariant subspaces.

\subsection{Equivalence}\label{subsection Equivalence}
Let $\lambda\in\rho+\Z$ and assume $0<\lambda<\rho$.

\begin{lemma}\label{lemma equivalent}
The $(\gf,K)$-modules $\Ec_{\lambda,K}^{\even}(X)$ and $\Ec_{\lambda,K}^{\odd}(X)$ are irreducible and
equivalent.
\end{lemma}

\begin{proof}
The assumption on $\lambda$ implies that $D_\lambda=\N_0$, and hence $U_\lambda$ is equal to
one of the two modules $\Ec_{\lambda,K}^{\even}(X)$ and $\Ec_{\lambda,K}^{\odd}(X)$,
depending on the parity of $\lambda-\rho$. For simplicity of exposition, let us
assume a specific parity, say even, of $\lambda-\rho$. Then $\Ec_{\lambda,K}^{\odd}(X)=U_\lambda$
is irreducible as seen in Section \ref{subsection Irreducibility}.

By Kostant's theorem \cite[Thm.~8]{Kos} an irreducible $(\gf,K)$-module, which contains the trivial
$K$-type, is uniquely determined up to equivalence by its infinitesimal character. Hence
$\Ec_{\lambda,K}^{\odd}(X)$ is equivalent to the irreducible subquotient of $\Ec_{\lambda,K}^{\even}(X)$
containing $\Ec_{\lambda,0}^{\even}(X)$. Since $\Ec_{\lambda,K}^{\odd}(X)$ and
$\Ec_{\lambda,K}^{\even}(X)$ contain the same $K$-types, all with multiplicity one, we conclude that
this subquotient is equal to $\Ec_{\lambda,K}^{\even}(X)$. The lemma is proved.
\end{proof}

\subsection{Conclusion} Assume $\re\lambda>0$.
Then $\Ec_{\lambda,K}(X)\cap L^2(X)=U_\lambda$ by  Lemma \ref{lemma square integrability}.
By definition $U_\lambda$ is non-zero if and only if
$\lambda-\rho\in\Z$. In Section \ref{subsection K-types in L2} we saw that
it consists of even functions on $X$ when $\lambda-\rho$ is odd, and vice versa.
Finally, irreducibility was seen in Section \ref{subsection Irreducibility}.
Thus the proof of Theorem \ref{thm discrete series} is complete.

Assume $\lambda\in\rho+\Z$ and $0<\lambda<\rho$. Then
$\Ec_{\lambda,K}^{\even}(X)$ and $\Ec_{\lambda,K}^{\odd}(X)$ are irreducible and
equivalent by Lemma \ref{lemma equivalent}.
One of them equals $U_\lambda$ and belongs to $L^2(X)$, whereas we have seen in
Lemma \ref{lemma square integrability}
that the other one is non-tempered.
This proves Theorem \ref{thm small lambda}.


\begin{thebibliography}{99}


\bibitem{AAR}  G. E. Andrews, R. Askey, and R. Roy, Special functions, Cambridge 1999.

\bibitem{BW} R. Beals and R. Wong, Special functions and orthogonal polynomials, Cambridge 2016.

\bibitem{THK} T. H. Koornwinder, {\it Lowering and raising operators for some special orthogonal polynomials}.
Contemp. Math. {\bf 417} (2006), 227–238.


\bibitem{Kos} B. Kostant, {\it On the existence and irreducibility of certain series of representations},
Bull. Amer. Math. Soc. {\bf 75} (1969), 627–642.

\bibitem{LNR} N. Limić, J. Niederle, and R. R\k{a}czka,
{\it Eigenfunction expansions associated with the second order invariant
operator on hyperboloids and cones, III.}
J. Math. Phys. {\bf 8} (1967), 1079–1093.


\bibitem{Norlund} N. E. Nörlund, {\it The logarithmic solutions of the hypergeometric equation},
Mat. Fys. Skr. Dan. Vid. Selsk. {\bf 2}, no. 5  (1963).

\bibitem{Rossmann}  W. Rossmann, {\it Analysis on Real Hyperbolic Spaces}, J. Funct. Anal. {\bf
30} (1978), 448–477.

\bibitem{vanDijk} G. van Dijk, {\it On a Class of Generalized Gelfand Pairs}, Math. Z. {\bf 193}
(1986), 581–593.



\end{thebibliography}
\end{document}